\documentclass[letterpaper, 10 pt, conference]{ieeeconf}  %

\IEEEoverridecommandlockouts                              %

\usepackage{amsmath} %
\usepackage{amssymb}  %

\usepackage{amsthm}
\usepackage{xcolor}
\usepackage{mathtools,stmaryrd}
\usepackage{empheq}
\usepackage[hidelinks]{hyperref}
\usepackage{algorithm}
\usepackage{algorithmic}
\let\labelindent\undefined
\usepackage{enumitem}
\usepackage{lipsum}

\newcommand{\ie}{\emph{i.e.}}

\newcommand{\smallconc}[2]{\begin{bsmallmatrix} #1 \\ #2 \end{bsmallmatrix}}

\renewcommand{\>}{\rangle}

\newcommand{\dbrak}[1]{\llbracket #1 \rrbracket}

\newcommand{\ox}{{\mathring{x}}}

\newcommand{\ow}{{\mathring{w}}}

\newcommand{\co}{\operatorname{co}}

\newcommand{\bfone}{\mathbf{1}}

\newcommand{\minimize}{\operatornamewithlimits{minimize}}

\newcommand{\R}{\mathbb{R}}

\newcommand{\calA}{\mathcal{A}}
\newcommand{\calB}{\mathcal{B}}

\newcommand{\calG}{\mathcal{G}}

\newcommand{\calM}{\mathcal{M}}

\newcommand{\calR}{\mathcal{R}}
\newcommand{\calS}{\mathcal{S}}

\newcommand{\calW}{\mathcal{W}}
\newcommand{\calX}{\mathcal{X}}

\newcommand{\bfw}{\mathbf{w}}

\newcommand{\bfI}{\mathbf{I}}

\newcommand{\ol}[1]{\overline{#1}}

\newcommand{\rmd}{\mathrm{d}}

\newcommand{\ringz}{\mathring{z}}

\theoremstyle{plain}
\newtheorem{proposition}{Proposition}
\newtheorem{lemma}{Lemma}
\newtheorem{theorem}{Theorem}

\theoremstyle{definition}
\newtheorem{definition}{Definition}
\newtheorem{example}{Example}

\theoremstyle{remark}
\newtheorem{remark}{Remark}

\newcommand{\nx}{{n_x}}
\newcommand{\nw}{{n_w}}

\newcommand{\ny}{{n_y}}
\newcommand{\Rx}{\R^\nx}
\newcommand{\Rw}{\R^\nw}

\newcommand{\Ry}{\R^\ny}

\newcommand{\Uff}{U_{\text{ff}}}

\newcommand{\normo}[1]{\dbrak{#1}}

\title{\LARGE \bf
Efficient Norm-Based Reachable Sets via \\ Iterative Dynamic Programming
}

\author{Akash Harapanahalli and Samuel Coogan$^{1}$%
\thanks{$^{1}$Akash Harapanahalli and Samuel Coogan are with the School of Electrical and Computer Engineering, Georgia Institute of Technology, Atlanta, GA, 30318, USA. \{aharapan,sam.coogan\}.gatech.edu}%
}

\begin{document}

\maketitle
\thispagestyle{empty}
\pagestyle{empty}

\begin{abstract}
In this work, we present a numerical optimal control framework for reachable set computation using \emph{normotopes}, a new set representation as a norm ball with a shaping matrix.
In reachable set computations, we expect to continuously vary the shape matrix as a function of time.
Incorporating the shape dynamics as an input, we build a \emph{controlled embedding system} using a linear differential inclusion overapproximating the dynamics of the system, where a single forward simulation of this embedding system always provides an overapproximating reachable set of the system, no matter the choice of \emph{hypercontrol}.
By iteratively solving a linear quadratic approximation of the nonlinear optimal hypercontrol problem, we synthesize less conservative final reachable sets, providing a natural tradeoff between runtime and accuracy. 
Terminating our algorithm at any point always returns a valid reachable set overapproximation.
\end{abstract}

\section{Introduction}

One way to assess the safety of a complex nonlinear system under uncertainty is to verify the specification for its forward reachable set.
For nonlinear systems, the reachable set is often too complex to determine in closed form---instead, various algorithmic methods have been developed to compute guaranteed overapproximations.
To combat conservativeness in these algorithmic strategies, different set representations have been developed, which have different tradeoffs between reachable set accuracy and computational efficiency. To name a few, there are zonotopes~\cite{althoff_introduction_2015}, constrained zonotopes~\cite{scott_constrained_2016}, polynomial zonotopes~\cite{kochdumper_sparse_2020}, hybrid zonotopes~\cite{bird_hybrid_2023}, ellipsotopes~\cite{kousik_ellipsotopes_2022}, and taylor models~\cite{chen_reachability_2015}.
For linear systems, any generalized star set~\cite{duggirala_parsimonious_2016} can be propagated using only $n+1$ simulations.
Beyond these generator-based approaches, there are many other frameworks for reachable set computation.
For instance, level set methods use the Hamilton-Jacobi equations~\cite{mitchell_level_2000,bansal_hamilton-jacobi_2017,bansal_deepreach_2021} to represent the reachable set as a level set of a solution to a partial differential equation.

From a dual perspective, one can evolve halfspaces in linear systems using the adjoint equation \cite{varaiya_reach_2000}.
This dual formulation was recently re-explored in~\cite{harapanahalli_parametric_2025}, where they consider general parametric set representations defined by the intersection of constraints, instead of ones generated by constrained combinations.
As shown in~\cite{harapanahalli_parametric_2025}, the parameters defining the constraints can be arbitrarily updated as long as the offset term compensates accordingly.
This observation is the basic backbone of this work---we directly optimize over the dynamics of these parameters to attempt to minimize the final volume of our overapproximating reachable set.

Computing reachable sets using norm balls is not a novel approach on its own, and there are many works applying results from contraction theory to compute such overapproximations.
For instance, \cite{maidens_reachability_2014,arcak_simulation-based_2018} use bounds on logarithmic norms to bloat or shrink norm balls along nominal trajectories. 
The works \cite{fan_locally_2016,fan_simulation-driven_2017} use interval analysis to overapproximate the Jacobian matrix to automatically bound the logarithmic norm.
More recently, \cite{harapanahalli_linear_2024} uses a more accurate interval matrix to bound the logarithmic norm by comparing specifically to the nominal trajectory. 
The advantage of these works is in the simplicity of computing the logarithmic norm, which is often known in closed form---for example, as a simple eigenvalue maximization in the $\ell_2$ case, or a maximum row/column sum computation in the $\ell_\infty$/$\ell_1$ cases \cite{desoer_feedback_1975,davydov_non-euclidean_2022}.
However, to our knowledge, existing works fix the norm along a real portion of the trajectory, using semi-definite programming to find a single shaping matrix along that segment.
In this work, instead of successively choosing fixed norms for segments of the trajectory using semi-definite programming, we consider the shaping matrix as a continuous curve we can dynamically control, and use iterative dynamic programming along the nominal trajectory to find smoothly varying shaping matrices.

\subsubsection*{Contributions}
In this work, we reformulate the reachable set computation problem as an optimal control problem.
First, we propose a simple constraint-based set representation called a \emph{normotope}, which is the sublevel set of a norm with a specified shaping matrix, center, and radius.
We build a \emph{controlled embedding system} for the original system using interval-based linear differential inclusions and logarithmic norms, with the observation that we may smoothly vary the shape matrix of the normotope as long as the offset term compensates accordingly---we call this the \emph{hypercontrol}.

Next, we pose the following question: given an initial set for the system, can we find a hypercontrol mapping through the embedding space such that final reachable set has minimal volume?
To address this problem, we define an optimal control problem to synthesize hypercontrol policies resulting in small terminal set volume.
We show how a dynamic programming algorithms like iLQR can trade extra computations to iteratively optimize the accuracy of the reachable set estimate with respect to the hypercontrol.
With two case studies, we show how (i) just a few cheap iterations can dramatically improve the reachable set accuracy, and (ii) that the algorithm can learn nontrivial policies which accumulate extra uncertainty earlier in the trajectory to shrink the final overapproximating reachable set.

\section{Background and Mathematical Preliminaries}
\subsection{Notations}

Let $\|\cdot\|$ denote a norm.
Let $I$ denote the identity matrix of appropriate dimension and $\bfone$ denote the vector of all ones of appropriate dimension.
For linear operator $A : \Rx\to\Ry$ between normed vector spaces, we set $\|A\|_{x\to y} = \sup_{x\in\Rx : \|x\|_{\Rx} = 1} \|Ax\|_{\Ry}$ to be the operator norm of $A$.
For linear operator $A:\Rx\to\Rx$, define the induced logarithmic norm $\mu(A) = \limsup_{h\searrow0} \frac{\|I + hA\|_{x\to x} - 1}{h}$.
$GL(n)$ is the set of invertible $n\times n$ matrices.
Let $\co\calS$ denote the convex hull of $\calS$.

\subsection{Nonlinear Systems and Reachable Sets}

We consider nonlinear dynamical systems of the form
\begin{align} \label{eq:nlsys}
    \dot{x} = f(t,x,w),
\end{align}
where $x\in\Rx$ is the state of the system, $w\in\Rw$ is a disturbance input to the system, and $f : \R_{\geq0}\times\Rx\times\Rw\to\Rx$ is a $C^1$ smooth function.

Under these assumptions, the system~\eqref{eq:nlsys} has a unique trajectory starting from any initial condition $x_0\in\Rx$ at $t_0$, under essentially bounded disturbance map $\bfw:[t_0,\infty)\to \Rw$ defined for some neighborhood of $t_0$. 
Let $\phi_f(t;t_0,x_0,\bfw)$ denote this trajectory.

Suppose we are given an initial set $\calX_0\subseteq\Rx$ at initial time $t_0$, and a compact disturbance set $\calW\subset\Rw$. Then we can define the (forward) \emph{reachable set} of the system as 
\begin{align*}
    &\calR_f(t;t_0,\calX_0,\calW) \\
    &:= \{\phi_f(t;t_0,x_0,\bfw) : x_0\in\calX_0,\, \bfw:[t_0,t_f]\to\calW\}.
\end{align*}
The reachable set is generally intractable to find in closed-form, both analytically and computationally, as it requires an infinite number of forward simulations of the system.
The goal of this work and many others in the literature~\cite{althoff_set_2021} is to efficiently compute a set \emph{overapproximation} $\ol\calR_f(t;t_0,\calX_0,\calW) \supseteq \calR_f(t;t_0,\calX_0,\calW)$.

An overapproximation of the true reachable set can verify several safety specifications for the system---for instance, reach-avoid problems can be posed as
\begin{align*}
    \underbrace{\ol\calR(t_f;t_0,\calX_0,\calW) \subseteq \calG}_{\text{reach $\calG$ at $t=t_f$}}, \  \underbrace{\ol\calR(t;t_0,\calX_0,\calW) \subseteq \calA^\complement \ \forall t\in[t_0,t_f]}_{\text{avoid $\calA$ for every $t$}}.
\end{align*}
In the coming sections, we propose a new set representation called a normotope, build a controlled embedding system whose trajectory provides a normotope overapproximating the true reachable set, and optimize over the choice of hypercontrol mappings in this embedding space.

\section{Normotopes: A Simple Constraint-Based Set Representation}

In this section, we introduce the normotope, some basic properties, and a simple recipe to compute the normotope reachable set of a linear system using the adjoint dynamics.

\subsection{Definition and Examples}

\begin{definition}
    Given a norm $\|\cdot\|$, a \emph{normotope} is the set 
    \begin{align*}
        \normo{\ox,\alpha,y} := \{x\in\Rx : \|\alpha (x - \ox)\| \leq y\},
    \end{align*}
    where $\ox\in\Rx$ is the \emph{center}, $\alpha\in GL(\nx)$ is the square invertible $\nx\times\nx$ \emph{shape matrix}, and $y>0$ is the \emph{offset}.
\end{definition}

Normotopes are a special case of a \emph{parametope}, the general constraint-based set representation introduced in \cite{harapanahalli_parametric_2025}, and are dual to basic ellipsotopes~\cite{kousik_ellipsotopes_2022} (letting $z = \tfrac{\alpha}{y}(x - \ox)$, $\normo{\ox,\alpha,y} = \{y\alpha^{-1}z + \ox : \|z\| \leq 1\}$, which is a basic ellipsotope with center $\ox$ and generator matrix $y\alpha^{-1}$).

\begin{example}[$\ell_2$ normotopes]
The normotope
\begin{align*}
    \normo{\ox,\alpha,y}_2 &= \{x\in\Rx : \|\alpha(x - \ox)\|_2 \leq y\} \\
    &= \{x\in\Rx : (x - \ox)^T \alpha^T\alpha (x - \ox) \leq y^2\}
\end{align*}
is equivalent to an ellipsoid centered at $\ox$, with weighting matrix $\alpha^T\alpha$ and radius $y^2$.
$\alpha\in GL(n)$ implies $\alpha^T\alpha \succ 0$.
\end{example}

\begin{example}[$\ell_\infty$ normotopes]
The normotope
\begin{align*}
    \normo{\ox,\alpha,y}_\infty &= \{x\in\Rx : \|\alpha(x - \ox)\|_\infty \leq y\} \\
    &= \{x\in\Rx : -y\bfone \leq \alpha(x - \ox) \leq y\bfone\} \\
    &= \{x\in\Rx : \alpha\ox - y\bfone \leq \alpha x \leq \alpha\ox + y\bfone\}
\end{align*}
is equivalent to the H-rep polytope given by $\{x : Hx \leq b\}$, $H = \smallconc{\alpha}{-\alpha}$, $b = \smallconc{\alpha\ox + y\bfone}{-\alpha\ox + y\bfone}$, with $2n$ faces and $2^n$ vertices.
\end{example}

\begin{example}[$\ell_1$ normotopes]
The normotope
\begin{align*}
    \normo{\ox,\alpha,y}_1 &= \{x\in\Rx : \|\alpha(x - \ox)\|_1 \leq y\} \\
    &= \{x\in\Rx : -y\bfone \leq S\alpha(x - \ox) \leq y\bfone\} \\
    &= \{x\in\Rx : S\alpha\ox - y\bfone \leq S\alpha x \leq S\alpha\ox + y\bfone\}
\end{align*}
where $S\in\R^{2^\nx\times\nx}$ is the matrix consisting of all $2^\nx$ possible choices of signs, \ie, the $k$-th row $S_k = (s^k_1,\dots,s^k_n)$ where $s^k_j\in\{+1,-1\}$. 
is equivalent to the H-rep polytope given by $\{x : Hx \leq b\}$, $H = \smallconc{S\alpha}{-S\alpha}$, $b = \smallconc{S\alpha\ox + y\bfone}{-S\alpha\ox + y\bfone}$, with $2^n$ faces and $2n$ vertices.
\end{example}

\begin{remark}[Nomenclature]
    We choose the name ``normotope'' over ``norm-ball'' to emphasize how we expect the shape matrix $\alpha$ to vary along reachable set trajectories.
\end{remark}

\subsection{The LTV Case}

In the undisturbed linear time-varying case (LTV) with a normotope initial set, the adjoint equation provides the true reachable set of the system as a normotope with fixed offset.

\begin{proposition} \label{prop:LTV_reach}
For the LTV system
\begin{align*}
    \dot{x}(t) = A(t) x(t),
\end{align*}
the true reachable set is given by
\begin{align*}
    \calR_f(t;t_0,\normo{\ox_0,\alpha_0,y_0}) = \normo{\ox(t),\alpha(t),y_0},
\end{align*}
where $t\mapsto\ox(t),\alpha(t)$ satisfy the following ODEs,
\begin{align*}
    \dot{\ox}(t) &= A(t) \ox(t), \quad
    \dot{\alpha}(t) = -\alpha(t) A(t).
\end{align*}
with initial conditions $\ox(t_0) = \ox_0$, $\alpha(t_0) = \alpha_0$.
\end{proposition}
\begin{proof}
We first recall how the flow map is given by $\phi_f(t;t_0,x_0) = \Phi(t;t_0) x_0$, where $\Phi$ is the state transition matrix satisfying the matrix ODE
\begin{align*}
    \dot{\Phi}(t;t_0) = A(t) \Phi(t;t_0), \quad \Phi(t_0;t_0) = \bfI,
\end{align*}
with the property $\Phi^{-1}(t;t_0) = \Phi(t_0;t)$.
Thus, the true reachable set is given by
\begin{align*}
    &\calR(t;t_0,\normo{\ox_0,\alpha_0,y_0})
    = \{\Phi(t;t_0)x : \|\alpha_0(x - \ox_0)\| \leq y_0\} \\
    &= \{z : \|\alpha_0(\Phi(t_0;t)z - \ox_0)\| \leq y_0\} \\
    &= \{z : \|\alpha_0\Phi(t_0;t)(z - \Phi(t;t_0)\ox_0)\| \leq y_0\} \\
    &= \normo{\Phi(t;t_0)\ox, \alpha_0\Phi(t_0;t), y_0} = \normo{\ox(t),\alpha(t),y_0},
\end{align*}
where $\alpha(t)$ satisfies the following,
\begin{align*}
    &\dot{\alpha}(t) = \alpha_0 \frac{\rmd}{\rmd t} \Phi(t;t_0)^{-1} 
    = - \alpha_0 \Phi(t;t_0)^{-1} \dot{\Phi}(t;t_0) \Phi(t;t_0)^{-1} \\
    &= -\alpha_0 \Phi(t_0;t) A(t) \Phi(t;t_0) \Phi(t;t_0)^{-1} = -\alpha(t) A(t),
\end{align*}
and $\ox(t)$ satisfies $\dot{\ox}(t) = A(t)\ox(t)$.
\end{proof}
We can interpret the ODE $\dot{\alpha}(t) = -\alpha(t) A(t)$ as each row $\alpha_j^T$ evolving according to the adjoint dynamics $\dot{\alpha}_j(t) = -A(t)^T \alpha_j(t)$.

In the nonlinear systems case, the true reachable set is generally not itself a normotope.
In the next section, we demonstrate how a similar embedding system approach can obtain guaranteed overapproximating normotope reachable sets by adding an appropriate dynamic compensation term to the offset $y$ using logarithmic and induced matrix norms. 

\section{Normotope Embeddings}

In Proposition~\ref{prop:LTV_reach}, we showed that the true reachable set in the linear time-varying case was given by a single forward simulation of another ODE (what we will refer to as an \emph{embedding system}), by evolving the center $\ox$ using the system dynamics and the rows of the shape matrix $\alpha$ using the adjoint dynamics. 
In this section, for nonlinear systems, we demonstrate how a similar \emph{controlled embedding} approach can be used to bound the behavior of any nonlinear dynamics.

\subsection{Dynamics Abstraction: Linear Differential Inclusions}

A \emph{linear differential inclusion} (LDI) \cite[p.51]{boyd_linear_1994} encompassing the error dynamics of the system~\eqref{eq:nlsys} is an inclusion where for every $x\in\calX$ and $w\in\calW$, for prespecified $\ox\in\calX\subseteq\Rx$ and $\ow\in\calW\subseteq\Rw$,
\begin{align} \label{eq:LDI}
    f(t,x,w) - f(t,\ox,\ow) \in \calM^x(t) (x - \ox) + \calM^w(t) (w - \ow),
\end{align}
where $\calM^x(t)\subseteq\R^{\nx \times \nx}$ and $\calM^w(t)\subseteq\R^{\nx \times \nw}$ are sets of matrices.
There are many ways to obtain matrix sets satisfying~\eqref{eq:LDI}.
For example, if $\overline{\co}\{\frac{\partial f}{\partial x}(t,x,w)\} \subseteq \calM^x(t)$ and $\ol{\co}\{\frac{\partial f}{\partial w}(t,x,w)\}\subseteq\calM^w(t)$ for every $(x,w)\in\calX\times\calW$ (where $\ol{\co}$ is the closed convex hull), then~\eqref{eq:LDI} holds for every $(x,w)\in\calX\times\calW$ \cite[Prop. 1]{harapanahalli_linear_2024}, \cite[p.55]{boyd_linear_1994}.

We briefly discuss an algorithmic method from~\cite[Cor. 1]{harapanahalli_linear_2024} to obtain the matrices $\calM^x$ and $\calM^w$ using interval analysis on the first partial derivatives of $f$. 
Suppose we are bounding a $C^1$ function $g:\R^n\to\R^m$, and we are given given an interval set $Z_1\times\cdots\times Z_n\subseteq\R^n$. 
Then the following implication holds, for fixed $\ringz\in Z_1\times\cdots\times Z_n$,
\begin{align*}
    &\frac{\partial g_i}{\partial x_j} (Z_1,\dots,Z_j,\mathring{z}_{j+1},\dots,\mathring{z}_n) \subseteq [\calM]_{ij}  \\
    &\implies g(z) - g(\mathring{z}) \in [\calM] (z - \mathring{z}),
\end{align*}
for every $z\in Z_1\times\cdots Z_n$~\cite[Cor. 1]{harapanahalli_linear_2024}, where the RHS is evaluated using interval matrix multiplication~\cite{jaulin_applied_2001}.
We can apply this procedure to the map $f(t,\cdot,\cdot):\Rx\times\Rw\to\Rx$ by defining an equivalent map $\hat{f}_t : \R^{\nx + \nw} \to\Rx$ to obtain the bound
\begin{align*}
    f(t,x,w) - f(t,\ox,\ow) \in [\calM(t)] \begin{bsmallmatrix} x - \ox \\ w - \ox \end{bsmallmatrix},
\end{align*}
and finally, we can split $[\calM(t)] \subseteq \R^{\nx \times (\nx + \nw)}$ into two interval matrices $[\calM^x(t)]\subseteq\R^{\nx\times\nx}$  and $[\calM^w(t)]\subseteq\R^{\nx\times\nw}$.

We simplify the expression by replacing the interval matrices $[\calM^x(t)]$ and $[\calM^w(t)]$ with the convex hull of a finite set of corners. For example, if the interval matrix $[\calM^x(t)]$ has $4$ non-singleton entries, then there are $2^4 = 16$ corners $M_i^x(t)$ to consider.
Thus, we obtain a bound of the following form,
\begin{align*}
    f(t,x,w) - & f(t,\ox,\ow) \in \\
    & \co\{M_i^x(t)\}_i (x - \ox) + \co\{M_j^w(t)\}_j (w - \ow).
\end{align*}
As we will see in Theorem~\ref{thm:main_normotope_thm} below, these finite sets will help us build a simple embedding system by evaluating logarithmic and induced matrix norms over each corner.

\subsection{Controlled Embedding System}

In Theorem~\ref{thm:main_normotope_thm}, we apply build a particularly well structured embedding system for the special case of a normotope set with a linear differential inclusion bounding the error dynamics to the center of the normotope.

\begin{theorem}[Controlled normotope embedding] \label{thm:main_normotope_thm}
Let $\|\cdot\|$ be a differentiable norm, the $\ell_1$-norm, or the $\ell_\infty$-norm.
Let $\mu$ be the induced logarithmic norm.
Define the following controlled embedding system,
\begin{subequations}\label{eq:embsys}
\begin{align} 
    \dot{\ox}(t) &= f(t,\ox,\ow), \\
    \dot{\alpha}(t) &= U(t,\ox,\alpha,y), \\
\begin{split} 
    \dot{y}(t) &= \left(\max_i \mu (U(t,\ox,\alpha,y) \alpha^{-1} + \alpha M^x_i(t) \alpha^{-1})\right) y \\
    &\quad\quad + \max_j \|\alpha M^w_j(t)\|_{w\to x},
\end{split}\label{eq:embsys:offset}
\end{align}
\end{subequations}
where $f(t,x,w) - f(t,\ox,\ow) \in \co\{M^x_i(t)\}(x - \ox) + \co\{M^w_i(t)\}(w - \ow)$ for every $(x,w)\in\partial\normo{\ox,\alpha,y}\times\calW$.
Then for every $t\geq 0$,
\begin{align*}
    \calR_f(t,\normo{\ox_0,\alpha_0,y_0},\calW) \subseteq \normo{\ox(t),\alpha(t),y(t)},
\end{align*}
where $t\mapsto(\ox(t),\alpha(t),y(t))$ is the trajectory of the system~\eqref{eq:embsys} under any locally Lipschitz map $U$.
\end{theorem}
\begin{proof}
((i) $\|\cdot\|$ is differentiable).
For $(x,w)\in\partial\normo{\ox,\alpha,y}\times\calW$, there is $M^x\in\co\{M^x_i\}_i$ and $M^w\in\co\{M^w_j\}_j$ such that
\begin{align*}
    &\xi(t,x,w) := \partial_{\alpha} g [U] + \partial_x g (M^x (x - \ox) + M^w (w - \ow)) \\
    &= \tfrac{\partial \|\cdot\|}{\partial z}\big|_{\alpha(x - \ox)} (U(x - \ox) + \alpha (M^x (x - \ox) + M^w (w - \ow))) \\
    &= \lim_{h\to0} \tfrac{\|\alpha(x - \ox) + hU(x - \ox) + hM^x(x - \ox) + hM^w(w - \ow)\| - \|\alpha(x - \ox)\|}{h} \\
    &\leq \limsup_{h\searrow 0} \tfrac{\|(I + h(U\alpha^{-1} + \alpha M^x\alpha^{-1})) \alpha(x-\ox)\| - \|\alpha(x - \ox)\|}{h} \\
    &\quad + \|\alpha M^w(w - \ow)\| \\
    &\leq \limsup_{h\searrow0} \tfrac{\|I + h(U\alpha^{-1} + \alpha M^x\alpha^{-1})\|_{x\to x} - 1}{h} \|\alpha(x - \ox)\| \\
    &\quad + \|\alpha M^w(w - \ow)\| \\
    &= \mu(U\alpha^{-1} + \alpha M^x\alpha^{-1}) y + \|\alpha M^w\|_{\operatorname{op}} \\
    &\leq \max_i\mu(U\alpha^{-1} + \alpha M^x_i \alpha^{-1}) y + \max_j\|\alpha M^w_j\|_{w\to x}.
\end{align*}
where the final inequality holds since $\mu$ and $\|\cdot\|_{w\to x}$ are convex functions~\cite[Lemma 2.3]{bullo_contraction_2023} with arguments linear in $M^x$ and $M^w$ respectively.
\cite[Cor. 1]{harapanahalli_parametric_2025} concludes the proof.

((ii) $\|\cdot\| = \|\cdot\|_\infty$) We first note that the normotope $\{\|\alpha(x - \ox)\|_\infty \leq y\}$ is equivalent to the symmetric H-rep polytope $\{-\hat{y} \leq \alpha(x - \ox) \leq \hat{y}\}$ where $\hat{y} = y\bfone\in\Rx$. 
Then for any $x$ satisfying $-\hat{y} \leq \alpha(x - \ox) \leq \hat{y}$ and $|\alpha_k(x - \ox)| = y$ for some $k$ (where $\alpha_k$ is the $k$-th row of $\alpha$),
\begin{align*}
    &\xi_k(t,x,w) := \dot{\alpha}_k (x - \ox) + \alpha_k(f(t,x,w) - f(t,\ox,\ow)) \\
    &= \dot{\alpha}_k (x - \ox) + \alpha_k (M^x (x - \ox) + M^w (w - \ow)) \\
    &= (\dot{\alpha}_k \alpha^{-1} + \alpha_k M^x \alpha^{-1}) \alpha (x - \ox) + \alpha_k M^w (w - \ow) \\
    &\leq \big((\dot{\alpha}_k \alpha^{-1} + \alpha_k M^x \alpha^{-1})_k \\
    &\quad + \textstyle\sum_{j\neq k} |(\dot{\alpha}_k \alpha^{-1} + \alpha_k M^x \alpha^{-1})_j|\big) y + \alpha_k M^w (w - \ow) \\
    &\leq \mu_\infty (U \alpha^{-1} + \alpha M^x \alpha^{-1}) y + \|\alpha M^w\|_{w\to x} \\
    &\leq \max_i \mu_\infty (U \alpha^{-1} + \alpha M^x_i \alpha^{-1}) y + \max_j \|\alpha M^w_j\|_{w\to x} 
\end{align*}
by definition of the $\mu_\infty$ log norm~\cite[Table 2.1]{bullo_contraction_2023}, and concluding with the same logic as part (i).
The resulting H-polytope embedding recovers the same trajectory as \eqref{eq:embsys} with the equivalence $\hat{y} = y\bfone$.

((iii) $\|\cdot\| = \|\cdot\|_1$) We first note that 
\begin{align*}
    \|x\|_1 &= \textstyle\sum_{i} |x_i| = \max_{s_j\in S} (s_j^T x) = \|Sx\|_\infty,
\end{align*}
where $S\in\R^{2^\nx\times\nx}$ is the matrix consisting of all $2^\nx$ possible choices of signs, \ie, the $k$-th row $S_k = (s^k_1,\dots,s^k_n)$ where $s^k_j\in\{+1,-1\}$. 
Thus, the normotope $\{\|\alpha(x - \ox)\|_1 \leq y\}$ is equivalent to the symmetric H-rep polytope given by $\{x : -\hat{y} \leq S \alpha (x - \ox) \leq \hat{y}\}$, where $\hat{y} = y\bfone$.
Let $S^+$ be a left inverse of $S$ ($S^+S = \bfI$).
Then for any $x$ satisfying $-\hat{y} \leq S\alpha(x - \ox) \leq \hat{y}$ and $|S_k\alpha(x - \ox)| = y$ for some $k$ (where $\alpha_k$ is the $k$-th row of $\alpha$),
\begin{align*}
    &\xi_k(t,x,w) = (S_k \dot{\alpha}) (x - \ox) \\
    &\quad + S_k\alpha (M^x (x - \ox) + M^w (w - \ow)) \\
    &= (S_k\dot{\alpha} \alpha^{-1} S^+ + S_k \alpha M^x \alpha^{-1} S^+) S\alpha (x - \ox) \\
    &\quad + S^k\alpha M^w(w - \ow) \\
    &\leq \mu_\infty(S(U\alpha^{-1} + \alpha M^x \alpha^{-1})S^+) y + \|S\alpha M^w\|_{\infty\to w},
\end{align*}
concluding with the logic of part (ii).
Finally, for any $A$, 
\begin{align*}
    \|SAS^+\|_\infty &= \sup_{z : \|z\|_\infty = 1} \|SAS^+z\|_\infty = \sup_{x : \|Sx\|_\infty = 1} \|SAx\|_\infty \\
    &= \sup_{x : \|x\|_1} \|Ax\|_1 = \|A\|_1,
\end{align*}
thus, plugging into their definitions, we see that $\mu_\infty(SAS^+) = \mu_1 (A)$ and $\|SA\|_{w\to\infty} = \|A\|_{w\to 1}$.
\end{proof}

In practice, we first overapproximate the normotope $\normo{\ox(t),\alpha(t),y(t)}$ with an interval subset, then apply the methodology from the previous section to obtain the matrices $\{M_i^x(t)\}_i$ and $\{M_j^w(t)\}_j$.
Finally, we note that the dynamics~\eqref{eq:embsys:offset} are not the only possible choice of embedding dynamics; in fact, any expression satisfying the hypotheses of~\cite[Cor. 1]{harapanahalli_parametric_2025} would constitute a valid embedding.
However, when the dynamics are abstracted using a linear differential inclusion as~\eqref{eq:LDI}, the expression~\eqref{eq:embsys:offset} is a natural choice with direct connections to contraction analysis~\cite{sontag_contractive_2010,davydov_non-euclidean_2022,bullo_contraction_2023}.

\begin{remark}[Connection to contraction theory]
The $U = 0$ case of Theorem~\ref{thm:main_normotope_thm} recovers the fixed norm results from~\cite{maidens_reachability_2014,fan_simulation-driven_2017,harapanahalli_linear_2024} for the norm with fixed shape matrix $\|x\|_{\alpha} = \|\alpha x\|$.
The introduction of the hypercontrol $U$ allows us to continuously vary the shape matrix along the trajectory of the normotope embedding system.
\end{remark}

\section{An Optimal Control Perspective Towards Reachable Set Computation}

In this section, we synthesize the policy $U$ from~\eqref{eq:embsys} using an optimal control formulation.
For clarity of presentation, we omit the disturbance input in this section and consider the nonlinear system
\begin{align*}
    \dot{x} = f(t,x),
\end{align*}
where $f : \R\times\Rx\to\Rx$ is $C^1$ differentiable, but we note that all results in this section can easily be generalized to the case with a compact disturbance.

As seen in Proposition~\ref{prop:LTV_reach}, for LTV systems the best choice of embedding dynamics is $\dot{\alpha} = -\alpha A(t)$. 
We choose hypercontrol policies of the form
\begin{align*}
    U(t,\ox,\alpha,y) = -\alpha \frac{\partial f}{\partial x}(\ox) + \Uff(t),
\end{align*}
where $\Uff$ is a feed-forward input to the embedding system. 
As discussed in~\cite{harapanahalli_immrax_2024}, closing the loop with the adjoint of the linearization cancels the first order expansion of the dynamics and generally provides reasonable results when the offset is small and the system is close to linear.
In the rest of this section, using the controlled embedding system from the previous section, we formulate an optimal control problem and use iterative dynamic programming to optimize the accuracy of our reachable set overapproximations.

\subsection{Cost Function Design}

Suppose we are given an initial set $\normo{\ox_0,\alpha_0,y_0}$ at $t=t_0$, and we are interested in verifying a goal specification at $t=t_f$.
For a goal specification, we are purely interested in the size of the reachable set at the final time, and therefore, we design an optimal control problem with a pure terminal cost.
The following Lemma recalls a standard result whereby maximizing the log-determinant of the shape matrix minimizes the volume of the parametope set.
\begin{lemma}
For any norm $\|\cdot\|$, a minimizer of the objecive function 
\begin{align} \label{eq:terminal_cost}
    \Phi(\ox,\alpha,y) = - \log\det (\alpha^T\alpha/y^2)
\end{align}
minimizes the volume of the normotope set 
\begin{align*}
    \operatorname{vol}(\normo{\ox,\alpha,y}).
\end{align*}
\end{lemma}
\begin{proof}
    The normotope set
    \begin{align*}
        \normo{\ox,\alpha,y} 
        &= \{x : \|\alpha(x - \ox)\| \leq y\} 
        = \{x : \|\tfrac{\alpha}{y}(x - \ox)\| \leq 1\}  \\
        &= \{(\tfrac{\alpha}{y})^{-1}z + \ox : \|z\| \leq 1\}
    \end{align*}
    is an affine image of the norm ball $\calB_1(0)$. Thus, the volume
    \begin{align*}
        \operatorname{vol}(\normo{\ox,\alpha,y}) = |\det((\tfrac{\alpha}{y})^{-1})| \operatorname{vol}(\calB_1(0))
    \end{align*}
    is inversely proportional to $|\det(\tfrac{\alpha}{y})|$. Since $\log((\cdot)^2)$ is a monotone function for positive inputs, we can maximize
    \begin{align*}
        \log(|\det(\tfrac{\alpha}{y})|^2) = \log (\det(\tfrac{\alpha^T}{y})\det(\tfrac{\alpha}{y})) = \log \det (\tfrac{\alpha^T\alpha}{y^2}),
    \end{align*}
    completing the proof.
\end{proof}

We implement normotopes and the terminal cost using JAX for Python, taking care to respect the stringent requirements of JAX, which provides us several crucial capabilities.
Normotopes are natively implemented as a Pytree node class, allowing us to use them as a valid JAX type for operations such as automatic differentiation.
For the terminal cost~\eqref{eq:terminal_cost}, we use \verb|jnp.linalg.cholesky| instead of composing \verb|log| and \verb|det| for numerical stability.

\subsection{Open-Loop Hypercontrol Synthesis: Iterative Linear Quadratic Regulator}

\newcommand{\ALGINPUT}{\textbf{Input:} }
\begin{algorithm}[tb]
\caption{Reach-iLQR} \label{alg:Reach-iLQR}
\textbf{Input:} initial set $X_0 = \normo{\ox_0,\alpha_0,y_0}$, time interval $[t_0,t_f]$ \\
\textbf{Parameters:} step size $\gamma>0$, regularizer $R\succ 0$, volume threshold $\Phi_\text{max}\in\R$
\begin{algorithmic}[1]
\STATE $\Uff^{(0)}(t) \gets 0$, $d(t) \gets 0$, $K(t) \gets 0$
\STATE $i\gets 1$
\REPEAT 
    \STATE Forward simulate under the updated control $\Uff^{(i)}$ from~\eqref{eq:control_update} to obtain $X^{(i)}(t)$ until $\Phi > \Phi_\text{max}$ or $t = t_f$.
    \STATE Compute terminal conditions \eqref{eq:terminal_conditions} using automatic differentiation on terminal cost $\Phi$ from \eqref{eq:terminal_cost}.
    \STATE Backward simulate Ricatti equations \eqref{eq:backward_pass} using automatic differentiation to obtain $d(t),K(t)$ from~\eqref{eq:d_K}.
    \STATE $i\gets i + 1$
\UNTIL {termination criterion reached}
\end{algorithmic}
\end{algorithm}
Suppose we are given an initial set $\normo{\ox_0,\alpha_0,y_0}$. 
In this section, we apply the continuous-time iterative linear quadratic regulator approach (iLQR)~\cite{lieskovsky_continuous-time_2025} to numerically optimize the following optimal control problem.
\begin{align}
    &\minimize_{\Uff:[t_0,t_f]\to\R^{\nx\times\nx}} \ \Phi(\ox(t_f),\alpha(t_f),y(t_f)) \label{eq:OCP} \\
    &\text{s.t. }\quad f(t,x) - f(t,\ox) \in \co\{M_i^x(t)\} (x - \ox), \nonumber \\
    &\begin{aligned}
    \dot{\ox}(t) &= f(t,\ox), \\
    \dot{\alpha}(t) &= -\alpha \tfrac{\partial f}{\partial x}(t,\ox) + \Uff(t) \\
    \dot{y}(t) &= \left(\max_{i} \mu(\Uff(t) \alpha^{-1} + \alpha (M_i^x(t) - \tfrac{\partial f}{\partial x}(t,\ox)) \alpha^{-1})\right) y \\
    \end{aligned} \nonumber \\
    & (\ox(t_0),\alpha(t_0),y(t_0)) = (\ox_0,\alpha_0,y_0) \nonumber
\end{align}

For the rest of this section, we write $X = \operatorname{vec}(\ox_0,\alpha_0,y_0) \in \R^{n_X} = \R^{\nx + n_x^2 + 1}$, and we abbreviate the embedding dynamics~\eqref{eq:embsys} as $\dot{X} = F(t,X,U)$. 

Next, we describe Algorithm~\ref{alg:Reach-iLQR}, which uses continuous-time iterative linear quadratic regulation (iLQR) to improve the feedforward term $\Uff$.
Suppose we are at timestep $i$, given $\Uff^{(i)}(t)$ and the corresponding embedding system trajectory $X^{(i)}(t)$ from forward simulating the embedding dynamics~\eqref{eq:embsys}.

\subsubsection{Backward Pass}
Since we have no running cost, the backwards pass simplifies to the following ODEs,
\begin{subequations} \label{eq:backward_pass}
\begin{align}
    -\dot{\sigma}(t) &= - \tfrac12 Q_U^T(t) R^{-1} Q_U(t), \\
    -\dot{s}(t) &= F_X^T (t) s(t) - Q_{UX}^T(t) R^{-1} Q_U(t), \\
    -\dot{S}(t) &= Q_{XX}(t) - Q_{UX}^T(t) R^{-1} Q_{UX}(t),
\end{align}
\end{subequations}
where $R\succ 0$ is some regularization matrix,
\begin{align*}
    Q_U(t) &= F_U^T(t) s(t), \\
    Q_{XX}(t) &= F_X^T(t) S(t) + S(t) F_X(t), \\
    Q_{UX}(t) &= F_U^T(t) S(t),
\end{align*}
and the following
\begin{align*}
    F_U^T(t) &= \frac{\partial F}{\partial U} (t,X^{(i)}(t),\Uff^{(i)}(t)), \\
    F_X^T(t) &= \frac{\partial F}{\partial X} (t,X^{(i)}(t),\Uff^{(i)}(t)),
\end{align*}
are computed using automatic differentiation in JAX.
The ODEs are initialized by automatically differentiating the terminal cost $\Phi$ with respect to the state $X$ as
\begin{subequations}\label{eq:terminal_conditions}
\begin{align} 
    \sigma(t_f) &= \Phi(X(t_f)), \\
    s(t_f) &= \frac{\partial \Phi}{\partial X}(X(t_f))^T, \\
    S(t_f) &= \frac{\partial^2 \Phi}{\partial X^2}(X(t_f)).
\end{align}
\end{subequations}
Finally, we obtain curves $d(t)$ and $K(t)$ as
\begin{align} \label{eq:d_K}
    d(t) = R^{-1} Q_U(t), \quad K(t) = R^{-1} Q_{UX}(t).
\end{align}
The key result from iLQR is that the variational system (first order expansion of the dynamics)
\begin{align*}
    \dot{\delta X} = F_X^T(t) \delta X(t) + F_U^T(t) \delta U(t)
\end{align*}
and the second order expansion of the value function
\begin{align*}
    V(x,t) \approx \sigma(t) + s(t)^T \delta X(t) + \tfrac12 \delta X(t)^T S(t) \delta X(t)
\end{align*}
satisfy the Hamilton-Jacobi-Bellman PDE with feedback controller $\delta U(t) = -d(t) - K(t)\delta X(t)$ (for the full derivation, please see~\cite{lieskovsky_continuous-time_2025}).

\subsubsection{Forward Pass}

From the backward pass, we obtain the curves $d(t)$ and $K(t)$. 
We then perform another forward sweep by updating $\Uff$ using the following (where $i$ has been incremented),
\begin{align} \label{eq:control_update}
    \Uff^{(i)}(t) = \Uff^{(i-1)}(t) - \gamma d(t) - K(t) (X^{(i)}(t) - X^{(i-1)}(t)),
\end{align}
where $\gamma>0$ is the step size, $X^{(i-1)}$ is the previous embedding trajectory under $\Uff^{(i-1)}$, and $X^{(i)}$ is the new embedding trajectory under $\Uff^{(i)}$. 
Theoretically, when $\gamma$ is small enough, the displacement $\delta X = X^{(i)} - X^{(i-1)}$ along the trajectory is small enough where the first order dynamics expansion and second order value expansion remain accurate.

While the forward pass always produces a valid reachable set overapproximation, it is often possible that the reachable set blows up before the prescribed final time $t_f$.
To combat this situation, we introduce a threshold $\Phi_\text{max}$, where we stop the forward simulation if $\Phi(X(t)) > \Phi_\text{max}$. 
When this happens, we proceed as if the final time $t_f$ was the time instance before the threshold was violated; in practice, we find that iterates will eventually tend towards policies which reach the final time.
The result is an algorithm where each iteration provides a valid overapproximate reachable tube, tending towards those with smaller terminal set volumes.

\begin{remark}[Comparison to the literature]
Previous approaches~\cite{maidens_reachability_2014,fan_simulation-driven_2017,harapanahalli_linear_2024} break the time interval $[t_0,t_f]$ into segments, solve intermediate semidefinite programs to optimize over the potential choices of norms along each individual segment, and ensure the reachable set containment by overapproximating the norm ball at the end of the previous segment with the new norm ball at the start of the next segment.
Algorithm~\ref{alg:Reach-iLQR} instead provides a continuous formulation where the hypercontrol smoothly varies the shaping matrix along the nominal trajectory, both avoiding successive norm overapproximations and allowing us to backpropogate the terminal set volume all the way to the initial timestep.
\end{remark}

\section{Examples}

In the following examples, we demonstrate how Algorithm~\ref{alg:Reach-iLQR} can (i) be used to quickly modify the reachable set to obtain drastically improve estimates in a short amount of time, and (ii) learn nontrivial open-loop policies for more difficult benchmarks.
For all ODE simulations, we use Euler integration with a step size of $0.01$. \footnote{The code is available at \url{https://github.com/gtfactslab/Harapanahalli_ACC2026}.}

\begin{figure}
    \centering
    \includegraphics[width=\columnwidth]{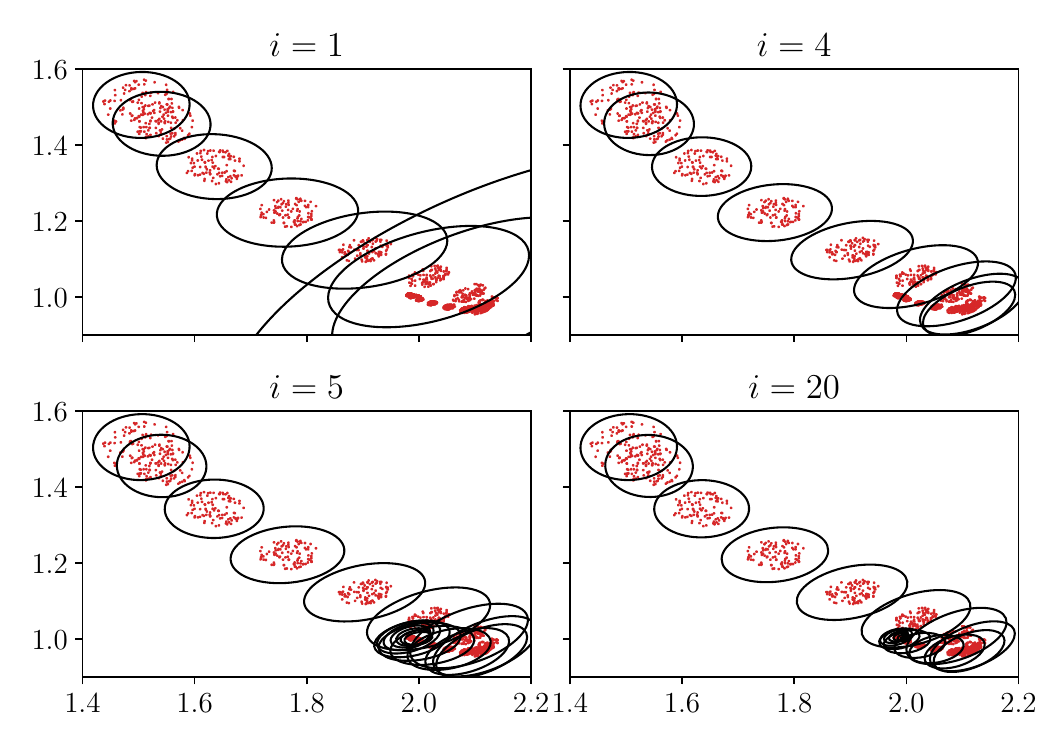}
    \includegraphics[width=\columnwidth]{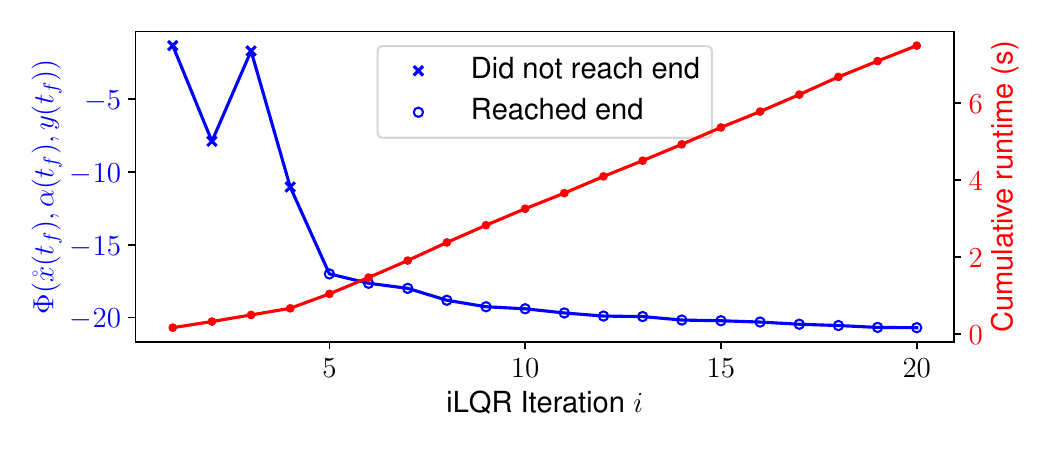}
    \caption{
    \textbf{Top:}
    Overapproximating reachable sets for the robot arm for $t\in[0,10]$ are plotted for various iteration numbers of Algorithm~\ref{alg:Reach-iLQR}, with Monte Carlo samples in red.
    \textbf{Bottom:} 
    The terminal cost $\Phi$ is plotted as a function of the iLQR iteration number in blue, while the cumulative runtime is in red. 
    After 5 iterations of iLQR, the overapproximate reachable set reaches the final timestep. 
    Iterating further, we can trade runtime for less overconservatism in the final reachable set. 
    }
    \label{fig:robot_ilqr}
\end{figure}

\begin{figure}
    \centering
    \includegraphics[width=\columnwidth]{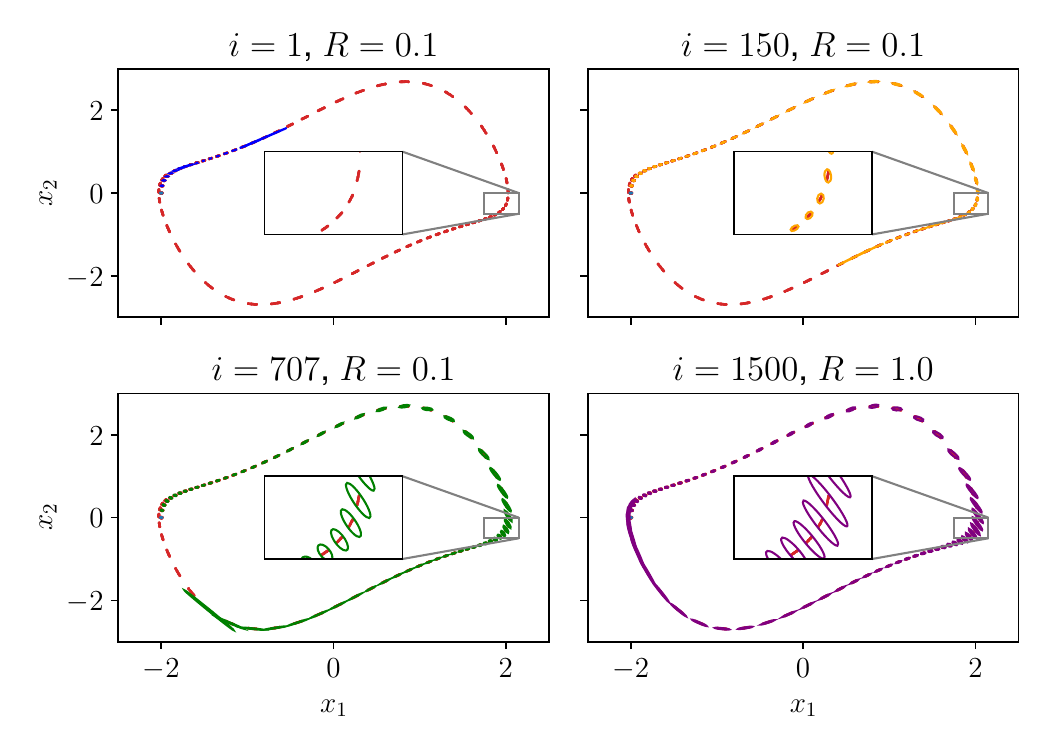}
    \includegraphics[width=\columnwidth]{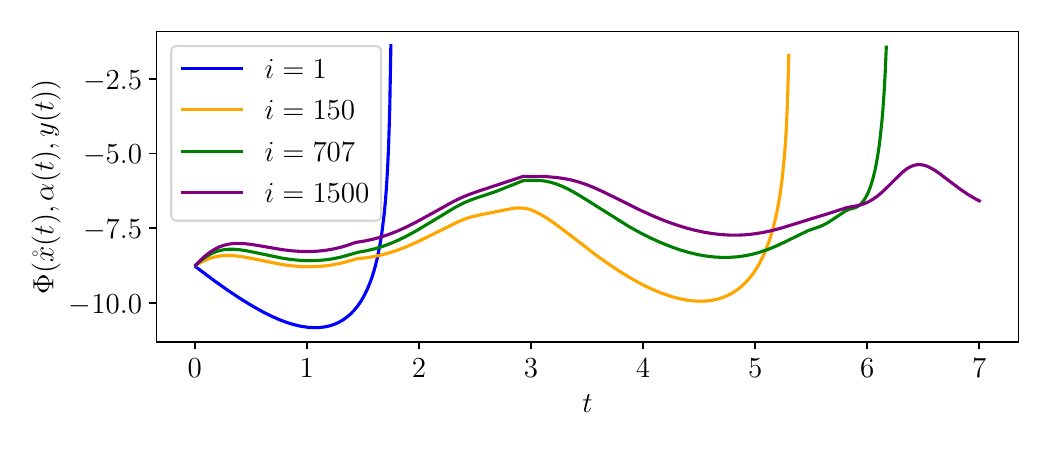}
    \caption{
    \textbf{Top:} Overapproximating reachable sets for the Van der Pol oscillator for $t\in[0,7]$ are shown for various iteration numbers of Algorithm 1, with Monte Carlo samples in red.
    \textbf{Bottom:} The function $\Phi(X(t))$ is plotted as a function of time for the same iteration numbers.
    While the pure adjoint dynamics ($i=1$) does well for small horizons, the error quickly accumulates over longer horizons and cannot reach the goal of $t_f=7$s.
    On the other hand, after 3000 iterations of iLQR, the algorithm qualitatively learns to accumulate extra error in earlier timesteps so it can reach and minimize the final timestep. 
    }
    \label{fig:vdp_ilqr}
\end{figure}

\subsection{Robot Arm}

Consider the dynamics of a robot arm from~\cite{angeli_characterization_2000,fan_locally_2016}
\begin{align*}
    \dot{q}_1 &= z_1, \quad \dot{q}_2 = z_2, \\
    \dot{z}_1 &= \tfrac{1}{mq_2^2 + ML^2/3}(-2mq_2z_1z_2 - k_{d_1}z_1 + k_{p_1}(u_1 - q_1)), \\
    \dot{z}_2 &= q_2 z_1^2 + \tfrac1m (- k_{d_2} z_2 + k_{p_2} (u_2 - q_2)),
\end{align*}
with $u_1 = 2$, $u_2 = 1$, $m=M=1$, $L=\sqrt{3}$, $k_{p_1} = 2$, $k_{p_2} = 1$, $k_{d_1} = 2$, $k_{d_2} = 1$.
We use the initial set $\normo{\ox_0,\alpha_0,y_0}_2$, where $\alpha_0 = P^{1/2}/0.1$ with $P$ and $\ox_0$ from~\cite{fan_locally_2016}, and $y_0 = 1$.
We note that this initial set has a radius $2.5$ times larger than the largest radius considered in the comparison example from~\cite{harapanahalli_linear_2024}, and thus a volume that is $39$ times larger.
We run Algorithm~\ref{alg:Reach-iLQR} with $t_f=10$, $\gamma=1$, $R = 20$, and $\Phi_\text{max} = -0.1$ for 20 iterations, taking a total of 7.5 seconds, and the results are shown in Figure~\ref{fig:robot_ilqr}.

\subsection{Van der Pol Oscillator} \label{ex:vanderpol}

Consider the Van der Pol oscillator written in usual state space coordinates as
\begin{align*}
    \dot{x}_1 = x_2,\quad \dot{x}_2 = -x_1 + \eta (1 - x_1^2)x_2,
\end{align*}
with nonlinearity $\eta = 1$. 
We consider the initial set $\dbrak{\ox_0,\alpha_0,y_0}_2$ for $\ox_0 = [-2\ 0]^T$, $\alpha_0 = 80I$, and $y_0 = 1$.

We split the iLQR iterations into two phases.
We first start with a small regularization of $R = 0.5I$ as an exploratory phase, keeping track of the best reachable set computation in all the iterates (the one that reaches the furthest before violation $\Phi > \Phi_\text{max}=-1.75$, with the smallest volume). 
After 750 iterations, we revert to the best iterate from the first phase ($i=707$) and use a larger regularizer $R = 5I$ for another 750 iterations to help reach the end.
All 1500 iterations took a total of 23.8 seconds to execute.

\section{Conclusion}
In this work, we presented an optimal control framework for synthesizing guaranteed overapproximate reachable sets with small terminal volume using iterative dynamic programming.
Our framework builds off the controlled parametric embedding approach described in \cite{harapanahalli_parametric_2025}, constructing a special case for \emph{normotope} sets using logarithmic and induced matrix norms.
In future work, we plan to use techniques from reinforcement learning to learn closed-loop policies using offline training to mimic the output of Algorithm~\ref{alg:Reach-iLQR}, which would be quick to evaluate for online reachable set computations.

\bibliographystyle{ieeetr}
\bibliography{references.bib}

\section*{Appendix}

\setcounter{subsection}{0}
\subsection{The LTV Case: Adjoint Satisfies Pontryagin's Principle}

As a sanity check, we verify that the adjoint equation satisfies Pontryagin's Minimum Principle (PMP) for the problem~\eqref{eq:OCP} in the LTV case.
Consider the linear system
\begin{align*}
    \dot{x} = A(t)x(t),
\end{align*}
and the embedding system from Theorem~\ref{thm:main_normotope_thm}, simplifying to
\begin{align*}
    \dot{\ox} = A(t) \ox(t), \ \dot{\alpha} = -\alpha(t) A(t) + \Tilde{U}(t)\alpha(t), \ \dot{y} = \mu (\Tilde{U}(t)) y,
\end{align*}
where we have made the variable substitution $\Tilde{U}(t) = (\alpha(t) A(t) + U(t))\alpha(t)^{-1}$.
Consider the volume terminal cost from~\eqref{eq:terminal_cost},
\begin{align*}
    \Phi(\ox,\alpha,y) = -\log\det(\alpha^T\alpha) + 2n\log y.
\end{align*}
The Hamiltonian of this problem is
\begin{align*}
    &H(t,\ox,\alpha,y,p_\ox,p_\alpha,p_y,\Tilde{U}) \\
    &= \<p_\ox,A\ox\> - \<p_\alpha, \alpha A\> + \<p_\alpha,\Tilde{U}\alpha\> 
     + \tfrac12 p_y y \lambda_{\max} (\Tilde{U}^T + \Tilde{U}),
\end{align*}
thus, the costate dynamics satisfy
\begin{align*}
    \dot{p}_\alpha &= -H_\alpha = p_\alpha A^T - \Tilde{U}^T p_\alpha \\
    \dot{p}_y &= -H_y = -\tfrac12 p_y \lambda_{\max} (\Tilde{U}^T + \Tilde{U})
\end{align*}
with terminal costates given by
\begin{align*}
    p_\alpha(t_f) = \tfrac{\partial\Phi}{\partial\alpha} = -2\alpha(t_f)^{-T}, \quad p_y(t_f) = \tfrac{\partial\Phi}{\partial y} = 2n/y(t_f).
\end{align*}
Set $\Tilde{U}^* = 0$.
Then $\dot{y} = 0$, $\dot{p}_y = 0$, and 
\begin{gather*}
    \frac{\rmd}{\rmd t} \alpha p_\alpha^T = \dot{\alpha} p_\alpha^T + \alpha \dot{p}_\alpha^T = -\alpha A p_\alpha^T + \alpha A p_\alpha^T = 0 \\
    \implies \alpha(t) p_\alpha(t)^T = \alpha(t_f) p_\alpha(t_f)^T = - 2 \bfI
\end{gather*}
Plugging into the Hamiltonian, 
\begin{align*}
    &H(\dots,\Tilde{U}) = \<p_\ox,A\ox\> - \<p_\alpha,\alpha A\> + \<p_\alpha,\Tilde{U}\alpha\> + \lambda_{\max} (\Tilde{U}^T + \Tilde{U}) \\
    &= H(\dots,\Tilde{U}^*) + n\lambda_{\max} (\Tilde{U}^T + \Tilde{U}) + \operatorname{tr}(p_\alpha^T\Tilde{U}\alpha) \\
    &= H(\dots,\Tilde{U}^*) + n\lambda_{\max} (\Tilde{U}^T + \Tilde{U}) - 2\operatorname{tr}(\Tilde{U}) \\
    &= H(\dots,\Tilde{U}^*) + n\lambda_{\max} (\Tilde{U}^T + \Tilde{U}) - \operatorname{tr}(\Tilde{U}^T + \Tilde{U}) \\
    &= H(\dots,\Tilde{U}^*) + \sum_{i=1}^n (\underbrace{\lambda_{\max} (\Tilde{U}^T + \Tilde{U}) - \lambda_i(\Tilde{U}^T + \Tilde{U})}_{\geq 0}).
\end{align*}
Thus, $H(\dots,\Tilde{U}^*)$ satisfies PMP.

\end{document}